\documentclass[a4paper,11pt,english]{amsart}

\usepackage[english]{babel}
\usepackage[latin1]{inputenc}
\usepackage[T1]{fontenc}

\usepackage{graphicx}
\usepackage{diagbox,amssymb}
\usepackage{verbatim,enumerate}

\newtheorem{thm}{Theorem}[section]
\newtheorem{thmintro}{Theorem}[]
\newtheorem{lemma}[thm]{Lemma}
\newtheorem{prop}[thm]{Proposition}
\newtheorem{conj}[thm]{Conjecture}
\newtheorem{cor}[thm]{Corollary}

{\theoremstyle{definition}

\newtheorem{rem}[thm]{Remark}}
\newcommand{\C}{\mathbb{C}}
\newcommand{\CC}{\mathcal{C}}
\newcommand{\CP}{\mathbb{C}P}
\newcommand{\RP}{\mathbb{R}P}

\newcommand{\R}{\mathbb{R}}
\newcommand{\Z}{\mathbb{Z}}

\newcommand{\Q}{\mathcal{Q}}
\newcommand{\x}{\underline{x}}
\newcommand{\y}{\underline{y}}

\renewcommand{\epsilon}{\varepsilon}
\newcommand{\N}{\mathcal{N}}
\newcommand{\Pic}{\text{Pic}}

\begin{document}

\title{Pencils of quadrics and Gromov-Witten-Welschinger invariants of $\CP^3$}

\author{Erwan Brugall\'e}
\address{Erwan Brugall\'e, CMLS, École polytechnique, CNRS, Université Paris-Saclay, 91128 Palaiseau Cedex, France}
\email{erwan.brugalle@math.cnrs.fr}

\author{Penka Georgieva}
\address{Penka Georgieva, Institut de Math\'ematiques de Jussieu - Paris Rive Gauche,
Universit\'e Pierre et Marie Curie, 
4 place Jussieu,
75252 Paris Cedex 5,
France}

\email{penka.georgieva@imj-prg.fr}

\subjclass[2010]{Primary 14P05, 14N10; Secondary 14N35, 14P25}
\keywords{Real enumerative geometry, Welschinger invariants,
Gromov-Witten invariants}
\thanks{P.G.  is supported by ERC grant STEIN-259118}

\begin{abstract} We establish a formula for the Gromov-Witten-Welschinger invariants of $\CP^3$ with mixed real and conjugate point constraints. The method is based on a suggestion by
J. Koll\'ar   that, considering pencils of quadrics, 
some real and complex enumerative
invariants of $\CP^3$ could be computed in terms of  enumerative invariants
of 
$\CP^1\times\CP^1$ and  of  elliptic curves.  \end{abstract}
\maketitle

Invariant signed counts of real rational curves with point constraints
in real surfaces and in many real threefolds were first defined by
J.-Y. Welschinger in \cite{Wel1} and \cite{Wel2}, respectively.  In the case
of surfaces, various methods for computation of these invariants are
established  \cite{Mik1, Sh8,Wel4, Br6b,HorSol12,IKS13,KhaRas13,Bru14}. In the case of threefolds,
methods for computation are developed only in the  extremal cases:
when all point constraints are real \cite{Br7}, when 
the number of real point constraints is minimal
\cite{Wel4,GeorZin13}, and
in the case of no constraints \cite{Sol08}. In this paper, we establish
a relation between the GWW invariants of $\CP^3$ and the GWW
invariants of $\CP^1\times \CP^1$, which allows the computation of the
former in terms of the latter. To our knowledge, Theorem
\ref{thm:main} provides the 
first systematic 
explicit 
computation of the Welschinger invariants of~$\CP^3$, and
more generally of a real algebraic variety of dimension 3,  when mixed real and conjugate point constraints are used.\\

Throughout the text, we equip $\CP^3$ and $\CP^1$ with their standard
real structure defined by the complex conjugation. We equip
$\CP^1\times \CP^1$ with the product of the  standard
real structure of $\CP^1$, in particular its real part is $\RP^1\times\RP^1$.
\\

Given  $d\in\Z_{>0}$ and $\x$ a generic real configuration of $2d$ points in
$\CP^3$, with $2l$ non-real points, 
we denote by $W_{\R P^3}(d,l)$ the corresponding Welschinger
 invariant counting with a sign the real rational curves of
degree $d$ in $\CP^3$ passing through all points in $\x$. Similarly,
given  $a,b\in\Z_{\ge 0}$ and $\x$ a generic real configuration of $2(a+b)-1$ points in
$\CP^1\times\CP^1$, with $2l$  non-real points, 
we denote by $W_{\RP^1\times\RP^1}((a,b),l)$ the Welschinger invariant counting with a sign the real rational curves of
 bidegree $(a,b)$ in $\CP^1\times \CP^1$ passing through all points in
$\x$.

\begin{thmintro}\label{thm:main}
For every odd positive integer $d$ and $l\in\{0,\ldots, d-1\}$, one has
$$W_{\R P^3}(d,l)=\sum_{\substack{a+b=d\\ 0\le a<b}} (-1)^a (d-2a) W_{\R P^1\times\R  P^1}\left((a,b),l\right).$$
\end{thmintro}
By symmetry reasons, 
G. Mikhalkin proved that $W_{\R P^3}(d,l)=0$ for even $d$ (see also Remark \ref{rem:vanishing}).
The invariants $W_{\R P^3}(d,d)$ have been computed in
\cite{GeorZin13}. 

\medskip
 We also establish a corresponding relation between the 
Gromov-Witten invariants of $\CP^3$ and $\CP^1\times \CP^1$. We denote by $GW_{\C P^3}(d)$ the count of degree $d$ rational curves in $\CP^3$ passing through~$2d$ generic points and by $GW_{\C P^1\times\C  P^1}\left(a,b\right)$ the count of rational curves of bidegree $(a,b)$ in $\CP^1\times \CP^1$ passing through $2(a+b)-1$ generic points. 
  
 \begin{thmintro}\label{thm:kollar}
For every  positive integer $d$, one has
$$GW_{\C P^3}(d)=\sum_{\substack{a+b=d\\ 0\le a<b}} (d-2a)^2 GW_{\C P^1\times\C  P^1}\left(a,b\right).$$
\end{thmintro}

The idea of the proof begins with a result of J. Koll\'ar, \cite[Proposition 3]{Kol14}, which can be summarised as follows. A non-degenerate
elliptic curve $C_0$ of degree 4 in $\CP^3$ generates a pencil of
quadrics  $\Q$  with 
base locus $C_0$. 
If the configuration  $\x$ of $2d$ points is contained in $C_0$,  every curve in degree $d$, passing through all points in $\x$,  is contained in a non-singular quadric of
$\Q$, where it is of bidegree $(a,b)$, with $a\ne b$. Furthermore, all
 such  quadrics  can be recovered by a computation in the Jacobian
 of $C_0$ and in particular there are exactly $(d-2a)^2$ of them.
 To complete the proof of
Theorem \ref{thm:kollar} from \cite[Proposition 3]{Kol14}, an additional
transversality argument is needed: one has to show  that a
configuration $\x$ contained in $C_0$ can be chosen so that  
 the number of the degree $d$ curves passing through $\x$ is indeed the
corresponding Gromov-Witten invariant of $\CP^3$. We prove that this
is indeed the case in Section \ref{sec:proof}. The proof of Theorem
\ref{thm:main} further requires a comparison of the signs of the curves
which enter in the definitions of the Welschinger  invariants of $\CP^3$ and $\CP^1\times\CP^1$. This is accomplished in Section \ref{sec:quadric}.\\

Theorems  \ref{thm:main} and \ref{thm:kollar} relate the enumerative
geometry of $\CP^3$ on one hand and of $\CP^1\times\CP^1$ and of
elliptic curves on the other hand. Both theorems can certainly 
be generalised
to $\CP^3$ blown up in a small number of points. It would be
interesting to understand further generalisations.\\

 The paper is organised as follows. 
We start by recalling the definition of the Welschinger invariants  of
$\CP^1\times \CP^1$ and  $\CP^3$ in
Section \ref{sec:def welsch}.  In Section \ref{sec:elliptic pencil} we review 
some standard facts concerning elliptic curves and pencils of quadrics
in $\CP^3$.   In Section
\ref{sec:quadric} we investigate properties of rational curves in a smooth quadric of
$\CP^3$  and establish the comparison of their two Welschinger signs.
  We combine results from Sections \ref{sec:elliptic pencil} and 
\ref{sec:quadric} in Section \ref{sec:proof} to provide a
transversality argument that allows us to deduce Theorem \ref{thm:kollar}
from \cite[Proposition 3]{Kol14} and we prove
Theorem \ref{thm:main}. We end this paper with   explicit
computations of Welschinger invariants of $\CP^3$ together with   further
qualitative results and comments about them.

\medskip
{\bf Aknowledgment: } We are grateful to B. Bertrand, I.
Itenberg, and G. Mikhalkin for useful discussions. 
We also wish to thank the anonymous referee for his valuable comments
on the first version of this text.

\section{Welschinger invariants}\label{sec:def welsch}
\subsection{Welschinger invariants  of  $\CP^1\times\CP^1$}

Given  $a,b\in\Z_{\ge 0}$ and $\x$
 a generic real configuration of $2(a+b)-1$   points in
$\CP^1\times\CP^1$, let $\R\CC'(\x)$ be the set of all real rational curves of
bidegree $(a,b)$ in $\CP^1\times \CP^1$ passing through all points in
$\x$. 
  One 
can     associate a sign $(-1)^{s_{\RP^1\times \RP^1}(C)}$  to each
real curve $C$ in $\R\CC'(\x)$ such that 
the sum
$$W_{\RP^1\times\RP^1}((a,b),l)=\sum_{C\in \R\CC'(\x)}(-1)^{s_{\RP^1\times \RP^1}(C)} $$
only depends on $a$, $b$,  and the number $l$ of pairs of complex
conjugated
points in $\x$; see
\cite{Wel1}. 
Since
$\x$ is generic, every curve $C\in \R\CC'(\x)$ 
is nodal.
A real node of $C$ is either the intersection of two real
branches of $C$, or the intersection of two complex conjugated
branches. 
The node is called hyperbolic in the former and elliptic in the latter case.
The  number  $s_{\RP^1\times \RP^1}(C)$ is defined 
 to be the number of elliptic real nodes of
$C$. \\

The parity of  
$s_{\RP^1\times \RP^1}(C)$ 
 can be interpreted in terms of a spin
structure on $\RP^1\times\RP^1$ as follows. 
Since a curve $C\in \R\CC'(\x)$ is nodal, it has a parametrisation 
  $f:\CP^1\to \CP^1\times \CP^1$ which is a real algebraic immersion.
A choice of a trivialisation of
 $T\RP^1$ induces a trivialisation
 $$
 \phi_0: T(\RP^1\times\RP^1)\to \RP^1\times\RP^1\times \R^2.
 $$
The canonical orientation and scalar product on $\R^2$ 
 induce, via $\phi_0$, an orientation and a
Riemannian metric on $T(\RP^1\times\RP^1)$.
In their turn, the latter  induce a
trivialisation $\phi$ and a Riemannian metric on the $\R$-vector bundle
$f_{|\RP^1}^*T(\RP^1\times\RP^1)$.
The tangent bundle $T\RP^1$ is naturally a subbundle of 
$f_{|\RP^1}^*T(\RP^1\times\RP^1)$;   let $\N$ be its orthogonal
subbundle.
Since $w_1(T\RP^1)=0$,
there exist a  nowhere vanishing 
smooth section $\sigma_T:\RP^1\to T\RP^1$. Let
 $\sigma_\N:\RP^1\to\N$ be a section  such that
 $(\sigma_T,\sigma_\N)$ is a
positive basis of $f_{|\RP^1}^*T(\RP^1\times\RP^1)$ and
 $N$ be the parity of the number of times   the basis $\phi\circ (\sigma_T,\sigma_\N)$
rotates around the canonical basis of $\R^2$ as one goes around $\RP^1$.
In other words, the number $N$ is the parity of the degree of the Gau{\ss} map of
$f(\RP^1)$. Note that it does not depend on the choice of the
trivialisation  of $T\RP^1$.

\begin{lemma}\label{lem:w2}
One has $s_{\RP^1\times \RP^1}(f(\CP^1))=N \mod 2$. 
\end{lemma}
\begin{figure}[h]
\begin{center}
\begin{tabular}{c}
\includegraphics[width=8cm, angle=0]{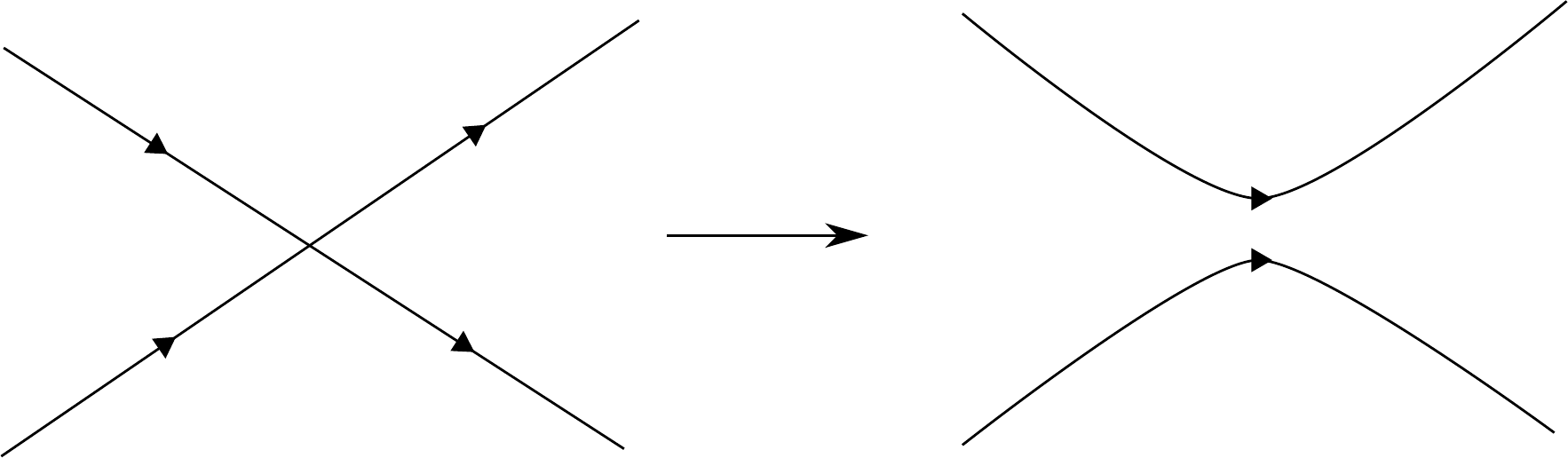}
\end{tabular}
\end{center}
\caption{Smoothing of $f(\RP^1)$}
\label{fig:node}
\end{figure}
\begin{proof}
We equip $f(\RP^1)$ with the orientation
induced by $\sigma_T$.
By  smoothing each node of $f(\RP^1)$ as depicted in
Figure \ref{fig:node}, we obtain a collection $\gamma$ of $n$ disjoint
oriented circles embedded
in $\RP^1\times\RP^1$. 
The sum of the Gau{\ss} index of all components of
 $\gamma$ is
 the Gau{\ss} index of $f(\RP^1)$ and we have
$$n=1+\kappa \mod 2, $$
where $\kappa$ is the number of hyperbolic nodes of $f(\CP^1)$.
Every 
embedded
circle  realising the zero class in $H_1(\RP^1\times\RP^1;\Z)$,
has Gau{\ss} index 
  $\pm1$.
Every 
embedded
circle,  which is non-trivial  in the homology of $ \RP^1\times\RP^1$,
is isotopic to 
a closed geodesic. In particular, it has Gau{\ss} index~0 and
represents a class 
$(p,q)\!\in\!H_1(\RP^1\times\RP^1;\Z)$, 
with $p$ and $q$ relatively prime.
Furthermore, since the components of the set~$\gamma$ do not
intersect, all non-trivial components must represent (up to orientation) the 
same 
class $(p,q) \!\in\! H_1(\RP^1\times\RP^1;\Z)$
and 
 the class of $f(\R P^1)$ can be written as $(mp,mq)$ for some $m\!\in\!\Z$.
Since the number $N$ is equal to the parity of the number of
homologically 
trivial components in $\gamma$, i.e. $n-m$ mod 2, we have 
\begin{center}
\begin{tabular}{ll}
$N=\kappa +1+m$&$ \mod 2$.  
\end{tabular}
\end{center}
\noindent
If $f(\C P^1)$ has bidegree $(a,b)$, by the adjunction formula it has exactly $(a-1)(b-1)$
nodes, $\kappa$ of which are hyperbolic. 
Thus,
\begin{center}
\begin{tabular}{ll}
$N=(a-1)(b-1) +1+m + s_{\RP^1\times \RP^1}(f(\CP^1))$&$ \mod 2$.  
\end{tabular}
\end{center}
\noindent
 Since $(mp,mq)=(a,b)\mod 2$, we have that
$m$ is even if and only if both $a$ and $b$ are even and so
$$s_{\RP^1\times \RP^1}(f(\CP^1))=N\mod 2,$$
 as announced.
\end{proof}

\subsection{Welschinger invariants of $\CP^3$}\label{sec:w3}

Given  $d\in\Z_{\ge 0}$ and $\x$
 a generic real configuration of $2d$   points in~$\CP^3$, let $\R\CC(\x)$ be the set of all real rational curves of
degree $d$ in $\CP^3$ passing through all points in
$\x$. 
  One 
can   again  associate a sign $(-1)^{s_{\RP^3}(C)}$  to each
real curve $C$ in $\R\CC(\x)$ such that 
the sum

$$W_{\R P^3}(d,l)=\sum_{C\in \R\CC(\x)}(-1)^{s_{\RP^3}(C)} $$
only depends on $d$ and the number $l$ of pairs of complex conjugated
 points in $\x$; see
\cite{Wel2}. 
We now recall the definition of $s_{\RP^3}(C)$. \\

We fix once and for all an orientation on $\RP^3$.
The projective space $\RP^3$ is  spin and
we may choose a trivialisation $\phi_0:T\RP^3\to \RP^3\times \R^3$ 
of its tangent bundle compatible with the chosen orientation.
The pullback by $\phi_0$ of the
canonical  Euclidean  scalar product on $\R^3$
provides a Riemannian metric
on $\RP^3$.
Let $f:\CP^1\to\CP^3$ be a real algebraic immersion.
The trivialisation of $T\RP^3$ induces a trivialisation and a
Riemannian metric on the $\R$-vector bundle $f_{|\RP^1}^*T\RP^3$.
The tangent bundle $T\RP^1$ is naturally a $\R$-subbundle of 
$f_{|\RP^1}^*T\RP^3$; let $\N_\R$ be its orthogonal
$\R$-subbundle. Choose an orientation of $\RP^1$ and a positive
orthonormal section $\sigma_T$ of $T\RP^1$.
Given 
a line $\R$-subbundle~$E$ of~$\N_\R$ with a non-vanishing smooth section
$\sigma_E$, such that $(\sigma_T,\sigma_E)$ is an orthonormal section of
$T\RP^1\oplus E$,   there exists  a unique choice of a section~$\sigma_\N$ of~$\N_\R$, such 
that $(\sigma_T,\sigma_E,\sigma_\N)$ is a positive orthonormal section of 
$f_{|\RP^1}^*T\RP^3$. Combining with the trivialisation $\phi_0$ of the latter
bundle, this section defines a loop in $SO_3(\R)$.
We define 
$$s(E)\in \{0,1\}$$ 
to be the number realised by this loop in 
 $\pi_1(SO_3(\R))=\{0,1\}$. 
 Note that $s(E)$ only depends on the isotopy class
of $E$ as an $\R$-subbundle of $\N_\R$ and on the homotopy class of  the restriction of $\phi_0$   to~$T\R P^3_{|f(\R P^1)}$.\\

Let $d$ be the degree of  $f(\CP^1)$ in $\CP^3$. 
The quotient $\N_\C$ of $f^*T\CP^3$ by $T\CP^1$ is a holomorphic vector
bundle over $\CP^1$ with first Chern class $4d-2$. 
Suppose   that the curve $f(\CP^1)$ is
 \emph{balanced}, i.e. 
$\N_\C$ is isomorphic to the holomorphic bundle
 $\mathcal O_{\CP^1}(2d-1)\oplus \mathcal O_{\CP^1}(2d-1)$.
In this case, there is a consistent way of choosing a line $\R$-subbundle
$L(f)$ of the $\R$-bundle $\N_\R$ as follows.
 Holomorphic  line subbundles of $\N_\C$ are in
one-to-one correspondence with  rational functions $F:\CP^1\to\CP^1$:
a fiber of such subbundle
over the point
$u$ has equation $w=F(u)z$, where $(u,z,w)\subset \C^3$ are local coordinates on
$\N_\C$ such that both
$(u,z)$ and $(u,w)$ are local coordinates on $\mathcal
O_{\CP^1}(2d-1)$. Moreover, a holomorphic line subbundle of
$\N_\C$ defined by the rational function $F$ has degree $2d-1-\deg F$
and is real if and only if $F$ is real.
In particular, up to  real isotopy, there exists a unique holomorphic
real line subbundle of $\N_\C$  of degree $2d-1$ and two 
holomorphic
real line subbundles of~$\N_\C$ of degree~$2d-2$, depending on whether
$F_{|\RP^1}$ is orientation preserving or not. Phrased differently,
the two (real isotopy classes of) holomorphic
real line subbundles of $\N_\C$  of degree~$2d-2$  are characterised by
the direction in which the real part of a   fiber rotates in  $\R\N_\C$.
The Riemannian metric on~$\RP^3$ identifies the $\R$-bundles $\N_\R$ and
$\R\N_\C$. In particular,
the section $\sigma_T$ of $T\RP^1$ together with the  orientation
of $\RP^3$ induce an orientation on the bundle $\R\N_\C$: given $u\in
\RP^1$, a positive 
basis of the fiber at $u$ of $f_{|\RP^1}^*T\RP^3$ is formed by
 $(\sigma_T(u), v_1,v_2)$ with $(v_1,v_2)$ a positive basis of $\R\N_\C$. 
We denote by $L(f)$ (resp. $\overline L(f)$) the isotopy class of the real part of the degree $2d-2$
  line subbundle of~$\N_\C$   whose real fibers rotate positively
  (resp. negatively)
in local holomorphic
coordinates on
$\N_\C$ defining a real holomorphic splitting  $\N_\C=\mathcal
O_{\CP^1}(2d-1)\oplus\mathcal O_{\CP^1}(2d-1)$. 
Since  $L(f )$ and $\overline L(f )$
differ by exactly one full rotation, which is a generator of $\pi_1(SO_2(\R))$, 
we have
$s(L(f )) \ne s(\overline L(f ))$.\\

For a generic real configuration $\x$ of points in $\RP^3$, 
every curve $C$ in $\R\CC(\x)$ is parametrised 
by a balanced immersion $f:\CP^1\to \CP^3$ and 
 the number  $s_{\RP^3}(C)$ is defined as
$$s_{\RP^3}(C)=s(L(f)).$$
The number  $s_{\RP^3}(C)$ is independent of the
parametrisation but 
in general
  depends on the
choice of a trivialisation $\phi_0$. 
%
In the remaining of this note,
we always assume
that $\phi_0$ is chosen so that $s_{\RP^3}(D)=0$ for a line $D$ in $\CP^3$.
\footnote{Such a choice is possible since every trivialization of $T\R P^3$ over the 2-skeleton extends to the 3-skeleton. The two homotopy classes of trivializations over the 2-skeleton correspond to different values of $s_{\R P^3}(D)$.}

\section{Elliptic curves and pencil of quadrics in
  $\CP^3$}\label{sec:elliptic pencil}
In this section we recall some known facts about the Picard group of complex and
real
 elliptic
curves, their torsion points, and their
 relation with pencils of quadrics in $\CP^3$.

\subsection{Torsion points of complex and real 
elliptic curves}\label{sec:elliptic}
Let $C_0$ be a  complex elliptic curve.
Recall that a choice of $p_0\in C_0$ induces an isomorphism 
$$\begin{array}{cccc}
\psi:&C_0 &\longrightarrow & \Pic_0(C_0)
\\ & p &\longmapsto & [p]-[p_0],
\end{array}$$
 which induces in its turn a group structure on $C_0$.
 Geometrically, writing $C_0$ 
 as the quotient of $\C$ by a full rank lattice $\Lambda$ for which $p_0$ is the
 orbit of $0$, the group structure induced by $\psi$ on $C_0$
is simply the group structure
 inherited from $(\C,+)$ by the quotient map.
This description allows to easily describe torsion points of order
$d$ on $C_0$: if $\Lambda=u\Z + v\Z$, with $u$ and $v$ two  complex
numbers, then the set of solutions of 
\begin{equation}\label{equ:torsion}
dp=0
\end{equation}
  is a group of order $d^2$ isomorphic to
$\Z/d\Z\oplus\Z/d\Z$ and generated by 
$\frac{1}{d}u$ and $\frac{1}{d}v$.\\

Analogously, 
the map $\psi$ induces the following series of isomorphisms $\Psi_d$
with  $d\in\Z$:
$$\begin{array}{cccc}
\Psi_d:&\Pic_d(C_0) &\longrightarrow & C_0
\\ & \sum_{i=1}^d \left[p_i\right] &\longmapsto & \psi^{-1}\left(\sum_{i=1}^dp_i - dp_0\right)
\end{array}$$
satisfying
$$\Psi_d(E)+\Psi_{d'}(E')= \Psi_{d+d'}(E+E').$$

Suppose now that $C_0$ is  real, with $\R C_0\ne \emptyset$, and that
$p_0\in\R C_0$. If $\R C_0$ is not connected, the connected component of $\R C_0$
containing $p_0$ is called the \emph{pointed component} of $\R C_0$. 
The real structure on $C_0$ induces a real structure on $\Pic_d(C_0)$ 
for every $d\in\Z$ and 
the maps $\Psi_d$ are all real maps (see \cite{GroHar81}).
Recall also (see \cite{Nat90}) that
$C_0$ can be expressed as 
$\C/\Lambda$ with the real structure inherited by the complex
conjugation on $\C$, where $\Lambda$ has one of the following forms:
\begin{itemize}
\item $\Lambda=u\Z + iv \Z$ with $u$ and $v$ two real numbers. In this
  case $\R C_0$ has two connected components: $\R/u\Z$ and
  $(\R +\frac{iv}{2})/u\Z$ (see Figure \ref{fig:real ell}a). When $d$
  is even, both
  connected components of $\R C_0$ contain exactly $d$ solutions of
  Equation $(\ref{equ:torsion})$. When $d$ is odd, Equation
  $(\ref{equ:torsion})$ has exactly $d$ real solutions, all located on
  the pointed component of $\R C_0$.

\medskip
\item $\Lambda= u\Z + \overline u \Z$ with $u$ a complex number. In
  this case $\R C_0=\R/(u+\overline u)\Z$ is connected (see Figure
  \ref{fig:real ell}b). Equation
  $(\ref{equ:torsion})$ has exactly $d$ real solutions for any $d$.
\end{itemize}
\begin{figure}[h!]
\centering
\begin{tabular}{ccc} 
\includegraphics[height=3.5cm, angle=0]{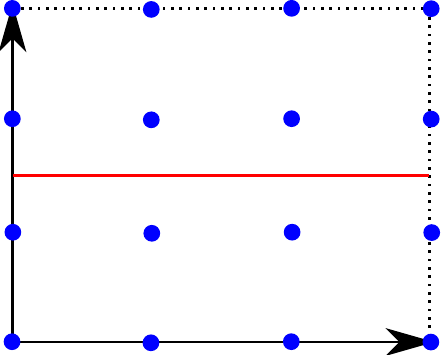}
\put(5, 2){$u$}
\put(-135, 95){$iv$}
\put(5, 52){$\R+\frac{iv}{2}$}
& \hspace{10ex} &
\includegraphics[height=3.5cm, angle=0]{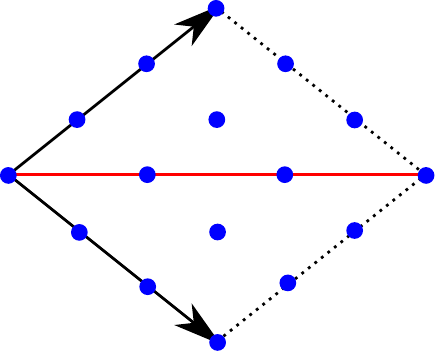}
\put(-80, 95){$u$}
\put(-80, 2){$\overline u$}
\put(5, 52){$\R$}
\\
\\ a) A maximal real elliptic curve&& b) A real elliptic curve with a
connected real part 
\end{tabular}
\caption{Uniformisation of real elliptic curves with a non-empty real
  part. The 
points represent the solutions of $3p=0$.}
\label{fig:real ell}
\end{figure}
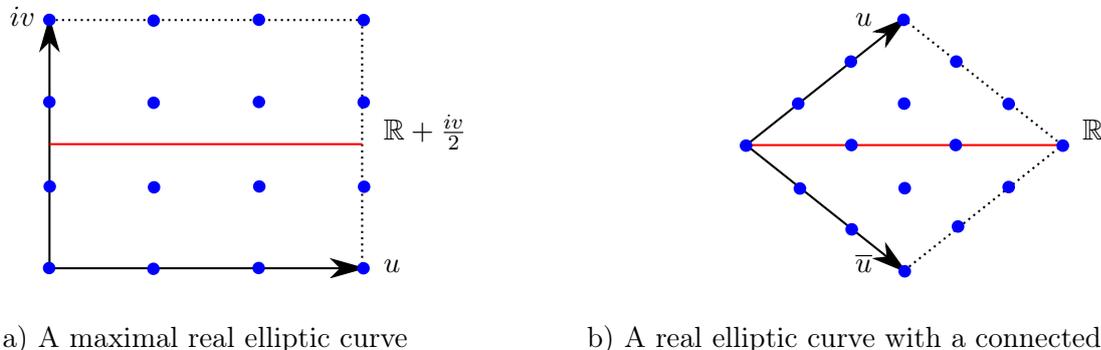

In this paper,  
all considered elliptic curves   are assumed to
  be equipped with a distinguished point~$p_0$, which is real if the
  curve is real. In particular, we always identify $\Pic_d(C_0)$ with $C_0$ 
via the map~$\Psi_d$.

\subsection{Pencils of quadrics}\label{sec:pencil}
Let $\Q$ be a pencil of quadrics in $\CP^3$, i.e. a line in the space
$\CP^{9}$ of quadrics in $\CP^3$. We assume that $\Q$
is generic enough so that the
base locus of $\Q$ is a non-degenerate elliptic curve $C_0$ of degree
4 in $\CP^3$. 
Conversely, every non-degenerate elliptic curve $C_0$ of degree~4 in~$\CP^3$
defines a pencil of quadrics with base locus $C_0$.
We denote by $h\in C_0\simeq \Pic_4(C_0) $ the hyperplane section class.\\

A non-singular quadric in $\CP^3$ is isomorphic to
$\CP^1\times\CP^1$ and admits a ruling by two families of lines in
$\CP^3$; let $D_1$ and $D_2$ be two lines representing the families. Since $C_0$ is of bidegree $(2,2)$ in the non-singular
quadrics of $\Q$, 
any such quadric defines two elements $E_i=D_i\cap C_0$ of $ C_0\simeq \Pic_2(C_0)$, 
 whose sum is $h$. Conversely, given $E\in\Pic_2(C_0)$,
 the union of all lines in
$\CP^3$, whose intersection with $C_0$ is a divisor defining $E$ or $h-E$, is 
 a quadric $Q_E$ in $\Q$.
Hence, the map 
$$\begin{array}{cccc}
\pi_{\Q}:&C_0\simeq \Pic_2(C_0) &\longrightarrow & \Q
\\  & E  &\longmapsto & Q_E
\end{array} $$
is a ramified covering of degree two. The ramification values of
$\pi_{\Q}$ correspond to singular quadrics in $\Q$, which are
all quadratic cones.  The corresponding
critical points of $\pi_{\Q}$ are the solutions of the equation $2E=h$ and   so any two
of them differ by a torsion point of order 2. Thus, $\Q$ contains
exactly~4 distinct singular quadrics, in accordance with the
 Riemann-Hurwitz formula.\\

Suppose now that $C_0$ is real with $\R C_0\ne\emptyset$. In this case, the pencil $\Q$ is real and the fiber
over a regular value $Q\in\R \Q$ of $\pi_\Q$ consists of two real
points (resp. two complex conjugate points) if and only if $\R Q$ is
homeomorphic to $S^1\times S^1$ (resp. $S^2$).
By  the description of the real torsion points of~$C_0$ given in Section \ref{sec:elliptic}, we have the following
possibilities  for the map $\pi_{\Q|\R C_0}$:
\begin{itemize}

\item  $\R C_0$ is not connected and $h$ is on the non-pointed
  component of $\R C_0$ (equivalently, both components of $\R C_0$ are
  non-trivial in $\pi_1(\RP^3)$):  none of 
the four singular fibers of
  $\Q$ are real (see Figure \ref{fig:real pencil}a),

\item  $\R C_0$ is not connected and $h$ is on the pointed
  component of $\R C_0$ (equivalently, both components of $\R C_0$ are
  trivial in $\pi_1(\RP^3)$):  the four singular fibers of
  $\Q$ are real (see Figure~\ref{fig:real pencil}b),

\item $\R C_0$ is connected: exactly two singular fibers of  
  $\Q$ are real (see Figure~\ref{fig:real pencil}c). 

\end{itemize}
\begin{figure}[h!]
\centering
\begin{tabular}{ccccc} 
\includegraphics[height=5.5cm, angle=0]{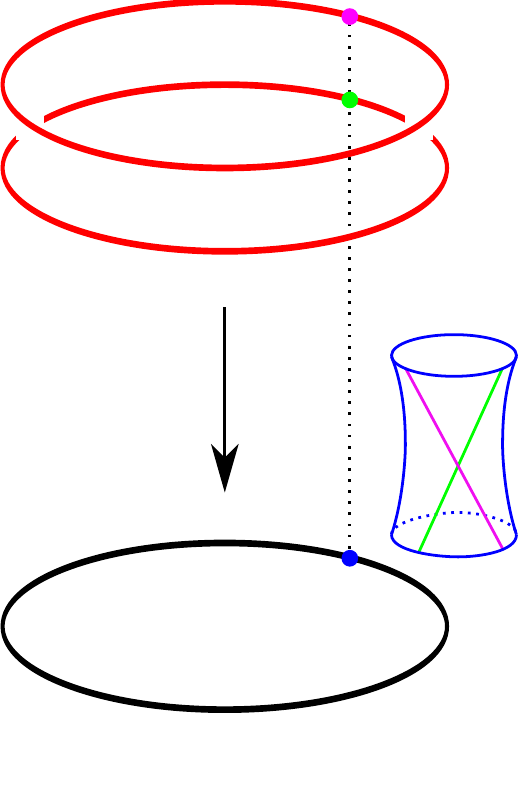}
\put(-15, 0){$\R\Q$}
\put(-5, 120){$\R C_0$}
\put(-80, 65){$\pi_\Q$}
&\hspace{5ex} 
& \includegraphics[height=5.5cm, angle=0]{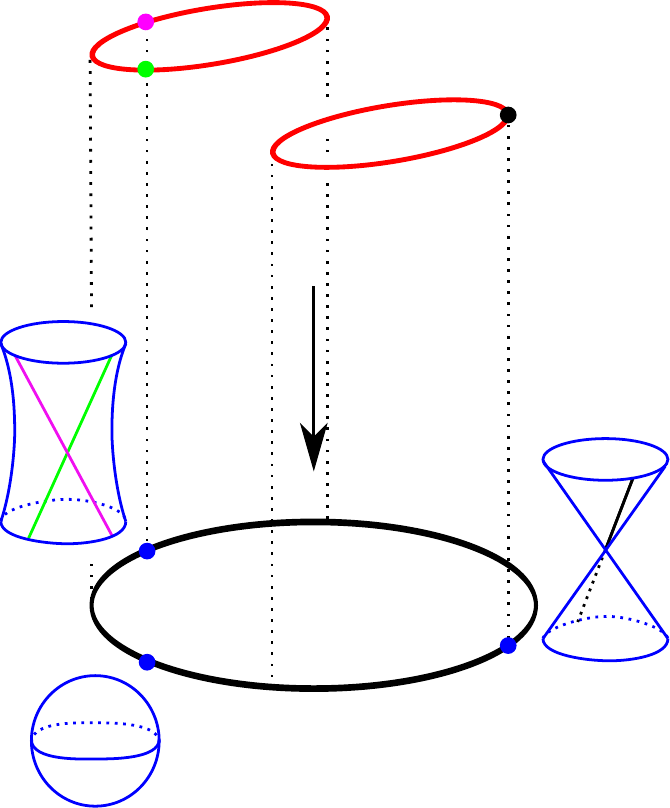}
&\hspace{2ex} 
&\includegraphics[height=5.5cm, angle=0]{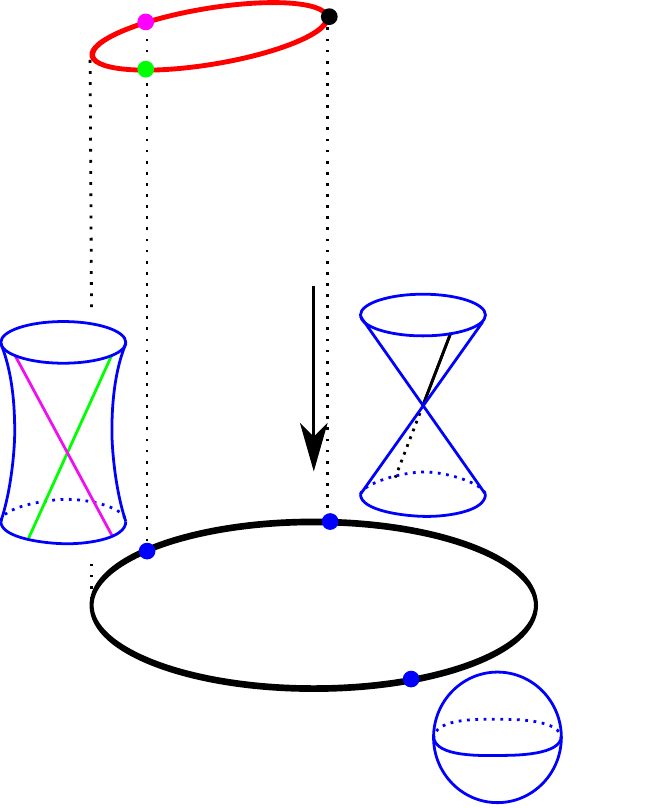}
\\
\\ a) && b)&& c) 
\end{tabular}
\caption{Real pencils of quadrics}
\label{fig:real pencil}
\end{figure}

\section{Rational curves on a smooth quadric of $\CP^3$}\label{sec:quadric}
In this section we study properties of rational curves in a smooth quadric in $\CP^3$ and establish a comparison of the signs of  a  real curve as an element of the quadric and of the projective space. \\

Throughout the section we fix a smooth quadric $Q$ of $\CP^3$. In particular,
$Q$ is isomorphic to $\CP^1\times\CP^1$ and we denote by $D_1$ and
$D_2$ two distinct intersecting lines in $Q$.
Given an algebraic immersion $f:\CP^1\to \CP^3$ contained in $Q$, there is an exact sequence of holomorphic vector bundles over $\CP^1$:
\begin{equation}\label{exseq} 0\to \N'_{\C} \to \N_\C \to f^*\N_Q\to 0, \end{equation}
where $\N'_{\C}$ is the quotient bundle $f^* TQ/T\CP^1$, $\N_{\C}$ is the quotient bundle $f^* T\CP^3/T\CP^1$,
and  $\N_Q=T\CP^3 |_{Q}/TQ$ is the
normal bundle of~$Q$ in~$\CP^3$. When $Q$ and $f$ are real, we also have the corresponding real bundles $\R\N_\C', \R\N_\C,$ and $\R\N_Q$ fitting in an exact sequence as above.

\begin{prop}\label{prop:balanced2}
Let $f:\CP^1\to \CP^3$ be an algebraic 
immersion such that $f(\CP^1)$ is contained in~$Q$, where it
has bidegree
$(a,b)$, with $a\ne b$. Then, $f$ is balanced.
\end{prop}
\begin{proof}
We want to prove that the normal bundle
 $\N_\C$ is isomorphic to the holomorphic bundle 
$\mathcal O_{\CP^1}(2d-1)\oplus \mathcal
O_{\CP^1}(2d-1)$, where $d=a+b$.
 By the adjunction formula,
the line bundle $\N'_\C $ has degree $2d-2$ and thus $f^*\N_Q$ has
degree~$2d$.
Hence, the map $f$ is balanced if and only if the sequence~(\ref{exseq}) does not
split. According to \cite[Theorem 4.f.3]{GrifHar83}, whose proof
extends to immersions, this sequence 
splits if and only if $f(\C P^1)$ is a complete intersection. Since 
$a\ne b$,  this is not the case.\\

We
briefly 
recall the main lines of the proof of 
\cite[Theorem 4.f.3]{GrifHar83}.
Suppose that 
 the above exact sequence splits and let $C_0\subset Q$ be an elliptic curve
of bidegree $(2,2)$ intersecting $f(\C P^1)$ transversely.
 Then, there exists a holomorphic section $\sigma$
 of $ \N_\C$, which vanishes  only at the $2d$ points of $f(\CP^1)\cap
 C_0$. If $\Q$ denotes the pencil of quadrics in $\CP^3$ with base
 locus $C_0$, 
 the section $\sigma$ corresponds to a first order
 deformation $f_\epsilon:\CP^1\to Q_\epsilon$
 of $f$ in the pencil $\Q$. Since $\sigma$
 vanishes   at the $2d$ points in $f(\CP^1)\cap C_0$,  the divisor
 class realised in $\Pic_{2d}(C_0)$ by 
 $f_{\epsilon}(\CP^1)\cap  C_0$
   is constant. 
On the other hand, as explained in Section \ref{sec:pencil},
the
quadric $Q_\epsilon$ in $\Q$ is determined by  the
 class $E_{1,\epsilon}$ realised by $D_{1,\epsilon}\cap C_0$ in
 $\Pic_2(C_0)$, where $D_{i,\epsilon}$ is the deformation in 
$\Q_\epsilon$ of $D_i$. 
Since $f_{\epsilon}(\CP^1)$ realises the class 
$aD_{1,\epsilon}+bD_{2,\epsilon}= (a-b)D_{1,\epsilon} + bH$ in
$\Pic(\CP^1\times\CP^1)$, where 
$H$ is the hyperplane section, 
$f_{\epsilon}(\CP^1)\cap C_0$ realises the class 
$(a-b)E_{1,\epsilon} + bh$ in $\Pic_{2d}(C_0)$. 
However, the class  $h$ is constant
along the deformation, whereas  the class $(a-b)E_{1,\epsilon}$ is not,  since $a-b\ne 0$.
 This is a contradiction.
\end{proof}

Suppose now that
 $Q$ is real with a real part
homeomorphic to $\RP^1\times\RP^1$ and that $D_1$ and $D_2$ are also
 real. Recall that, for a real algebraic immersion $f:\C P^1\rightarrow \C P^3$, the   isotopy classes 
 $L(f)$ and~$\overline L(f)$ and the numbers $s(E)$, for an orientable $\R$-subbundle $E\subset \R\mathcal{N}_\C$,
have been defined in Section \ref{sec:w3}.

\begin{prop}\label{prop:lines} 
Let  
$\R\mathcal{N}'_i$
be the real part of the normal bundle of $D_i$ in $Q$, 
considered as a subbundle of the real part of its normal bundle in $\C P^3$.
 Then, one has $s(\R\mathcal{N}'_1)\ne s(\R\mathcal{N}'_2)$.
\end{prop}
\begin{proof}
Given a line $D$ in $\CP^3$, the holomorphic line subbundles of degree 1 of
its normal bundle in~$\CP^3$
 correspond precisely to the planes of $\CP^3$ containing
$D$. 
In particular, the normal bundle~$\R\mathcal{N}'_i$ realises the  
isotopy class 
$L(D_i)$
 if and only if it
 rotates 
in $\RP^3$ 
around $\R D_i$ in the positive direction. Since $\R\mathcal{N}'_1$
and $\R\mathcal{N}'_2$ do not rotate in the same direction, the result
follows.
\end{proof}
 
We will say that $(D_1,D_2)$ is \emph{the positive
  basis} of $H_2(Q;\Z)$ if $s(\R\mathcal{N}'_1)=0$.

\begin{cor}\label{prop:eps-eps'} 
Let $f:\CP^1\to \CP^3$ be a real algebraic 
immersion such that $f(\CP^1)$ is contained in~$Q$, where it
has a bidegree
$(a,b)$ in the positive basis, with $a\ne b$. Then, one has
$$s(\R\N'_\C)= s_{\RP^1\times\RP^1}(f(\CP^1)) +b.$$
\end{cor}
\begin{proof}
The proof is similar to the proof of Lemma \ref{lem:w2}.
Equip $\RP^1$ with some orientation and smooth
 each node of $f(\RP^1)$, as depicted in
Figure \ref{fig:node}, in order to
 obtain a collection $\gamma$ of oriented circles embedded
in $\RP^1\times\RP^1$. 
The loop in $\pi_1(SO_3(\R))$ defined by  $\R\N'_\C$ is freely
homotopic to the product of the loops   in $\pi_1(SO_3(\R))$ defined by
$T\R Q_{|\gamma_i}/T\gamma_i$ for $\gamma_i$ ranging over elements of $\gamma$. 
Any loop in $\gamma$ realising the 0 class (resp. the class
$p[\R D_1] + q[\R D_2]$ with gcd$(p,q)$=1) in
$H_1(\RP^1\times\RP^1;\Z)$
defines a non-trivial loop (resp.  $q$ times the non-trivial loop) 
in  $\pi_1(SO_3(\R))$.
Now the end of the proof is similar
to the end of
 the proof of Lemma \ref{lem:w2}.
\end{proof}

\begin{prop}\label{prop:L or not}
Let $f:\CP^1\to \CP^3$ be a real algebraic 
immersion such that $f(\CP^1)$ is contained in $Q$, where it
has  a bidegree
$(a,b)$ in the positive basis, with $a\ne b$. Then, the
 holomorphic real line
subbundle 
$\R\N'_\C$ of~$\R\N_\C$ 
realises the real isotopy class
 $L(f)$ if and
only if $a>b$.
\end{prop}
\begin{proof}
By convention, we have 
$s_{\RP^3}(D)=0$ for a real line $D$ in $\CP^3$. 
Hence, according to 
Proposition~\ref{prop:lines}, 
 the proposition is true for $(a,b)=(1,0)$ and $(a,b)=(0,1)$. 

Recall how  the two   isotopy classes $L(f)$ and
$\overline L(f)$ are characterised: in local holomorphic
coordinates on
$\N_\C$ defining a real splitting  $\N_\C=\mathcal
O_{\CP^1}(2d-1)\oplus\mathcal O_{\CP^1}(2d-1)$, the classes
$L(f)$ and~$\overline L(f)$  rotate in
different directions in $\R\N_\C$. The positive direction of rotation
is determined by an orientation of $\R P^1$ and of $\R P^3$; the class
$L(f)$ is by definition the one which rotates in the positive
direction.
Hence to prove  the proposition, it is enough
 to find a point $u_0\in\RP^1$ 
and a real holomorphic subbundle $E$ of $\N_\C$ of degree $2d-1$ such that
\begin{itemize}
\item the fibers
of $E$ and $\N'_\C$ over $u_0$ coincide and
\item  we can determine the mutual position of the real parts of the fibers of these
  bundles over a point $u\in\RP^1$ in a neighbourhood of $u_0$.
\end{itemize}

Let 
$C_0\subset Q$ be a real elliptic curve of bidegree $(2,2)$,
with $\R C_0\ne 0$, that intersects $f(\CP^1)$ transversely  at
some point $p_0=f(u_0)$ with $u_0\in \RP^1 $.
 We denote by $\Q$ the real 
 pencil of quadrics defined by $C_0$. In particular, $\Q$ is the
 realisation of a first order real deformation $Q_\epsilon$ of the
 real quadric $Q$ in $\CP^3$. Denoting by $D_{i,\epsilon}$ the deformation of $D_i$ in
 $Q_\epsilon$ and by $E_{i,\epsilon}\in\Pic_2(C_0)$ the class
 realised by $D_{i,\epsilon}\cap C_0$, recall that
\begin{equation}\label{equ:diff}
\frac{d E_{i,\epsilon}}{d \epsilon}|_{\epsilon=0}\ne 0.
\end{equation}

Let $f_\epsilon$ be a first order real
 deformation of $f$ in the pencil $\Q$ such that $f_\epsilon(\CP^1)$
 passes through  the $2d-1$ points of $f(\CP^1)\cap
 C_0\setminus\{p_0\}$ for all $\epsilon$.
This deformation corresponds to a non-null real holomorphic section 
$\sigma:\CP^1\to \N_\C$ that vanishes  on  $f^{-1}\left( C_0\setminus\{p_0\}\right)$.
Recall that,
since $a\ne b$,  the class realised by $f_\epsilon(\CP^1)\cap C_0$ in
$\Pic_{2d}(C_0)$ is not constant. In particular, $\sigma(u_0)\ne 0$.
Let $E$ be the real holomorphic subbundle  of $\N_\C$ of degree
$2d-1$ whose fiber over $u_0$ is the line generated by~$\sigma(u_0)$. 

\medskip
{\bf Claim 1:} $\sigma$ is a section of $E$. Indeed, $\sigma$ induces
a holomorphic section of the bundle $\N_\C/E$ that vanishes at 
the $2d$ points of 
$f^{-1}\left(C_0\right)$. Since the latter bundle has
degree $2d-1$, this induced section must be the null section i.e. $\sigma(u)\in E$ for all $u\in\CP^1$.

\medskip
{\bf Claim 2:} the fibers of $E$ and $\N'_\C$ over $u_0$ coincide.
Indeed, the vector
$\sigma(u_0)$ corresponds to the preimage
of the deformation of $p_0$ in
$f_\epsilon(\CP^1)\cap C_0$. 
Since this deformation stays in $C_0$ by definition  and   $C_0$
sits in $Q$, we must have 
$\sigma(u_0)\in  \N'_\C$.


\medskip
{\bf Claim 3:} the   isotopy class realised by $\R\N'_\C$ is determined by the
direction of the vector $\sigma(u_0)$.
 Let us denote by $\sigma_Q$ the holomorphic section of $T\CP^3/T Q$
 corresponding to the deformation $\Q$. Since both $\RP^3$ and $\R Q$
 are orientable manifolds, we may 
also fix  a smooth nowhere vanishing section $\lambda_Q$ of
 the $\R$-vector bundle  $T\RP^3/T\R Q$ (see Figure \ref{fig:L or not}a for a
 local picture at $p_0$). 
The section~$\sigma_Q$ vanishes along $C_0$ and we denote by $\R Q_+$
the region of $\R Q\setminus \R C_0$, where $\sigma_Q$ and $\lambda_Q$
have the same direction (see Figure \ref{fig:L or not}b). The choices
of $\lambda_Q$ and $\sigma_Q$ induce an
orientation on the real part of our source curve: we orient $\RP^1$ so
that $f(\RP^1)$ points toward $\R Q_+$ at $f(u_0)$. 
The choice of a Riemannian metric on $\RP^3$ identifies the $\R$-bundle
$f^*(T\RP^3/T \R Q)$ 
with the orthogonal of $\R \N'_\C$ in $\R \N_\C$. With this
identification, the vector $\sigma(u)$ for $u\in\RP^1$ close enough to
$u_0$  decomposes as
$$\sigma(u)=  g_1(u)\sigma(u_0)  + g_2(u)\lambda_Q(u),$$
where $g_1$ is a smooth function
with $g_1(u_0)=1$ and $g_2$ is a smooth  function  vanishing at~$u_0$  and
positive for $u>u_0$ (the local ordering of $\RP^1$ at $u_0$ is given
by the orientation of $\RP^1$). In other words, the choice of
$\lambda_Q$ and $\sigma_Q$ determine an orientation of the fiber of $\R\N_\C$ over
$u_0$ together with a half-plane $\Pi\subset \R\N_{\C| u_0}\setminus
\R\N'_{\C| u_0}$  containing $\sigma(u)$ when $u>u_0$ (see Figure
\ref{fig:L or not} c, d).  Now
clearly, 
the direction in which $\R \N'_\C$ rotates with respect to $\R E$ depends only on the direction of $\sigma(u_0)$.
\begin{figure}[h!]
\centering
\begin{tabular}{cc} 
\includegraphics[height=4cm, angle=0]{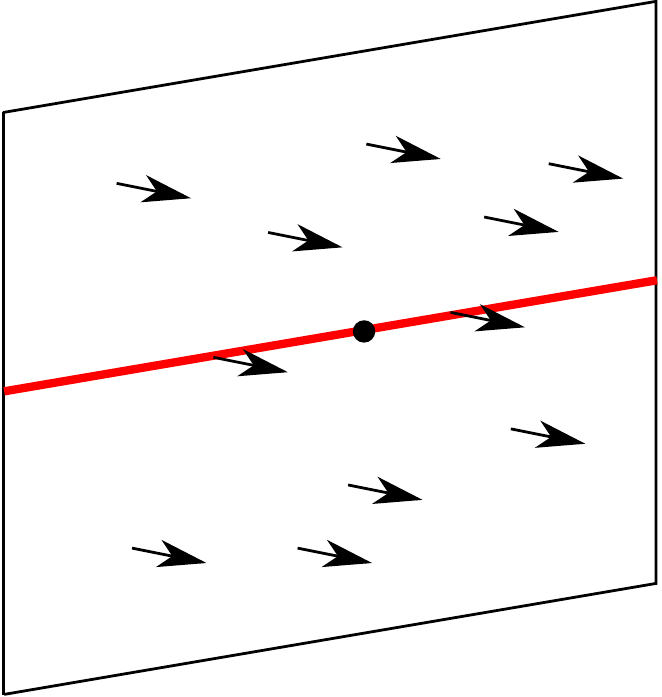}
\put(-110, 105){$\R Q$}
\put(5, 70){$\R C_0$}
\put(-50, 50){$p_0$}
\put(-80, 70){$\lambda_Q$}
& \includegraphics[height=4cm, angle=0]{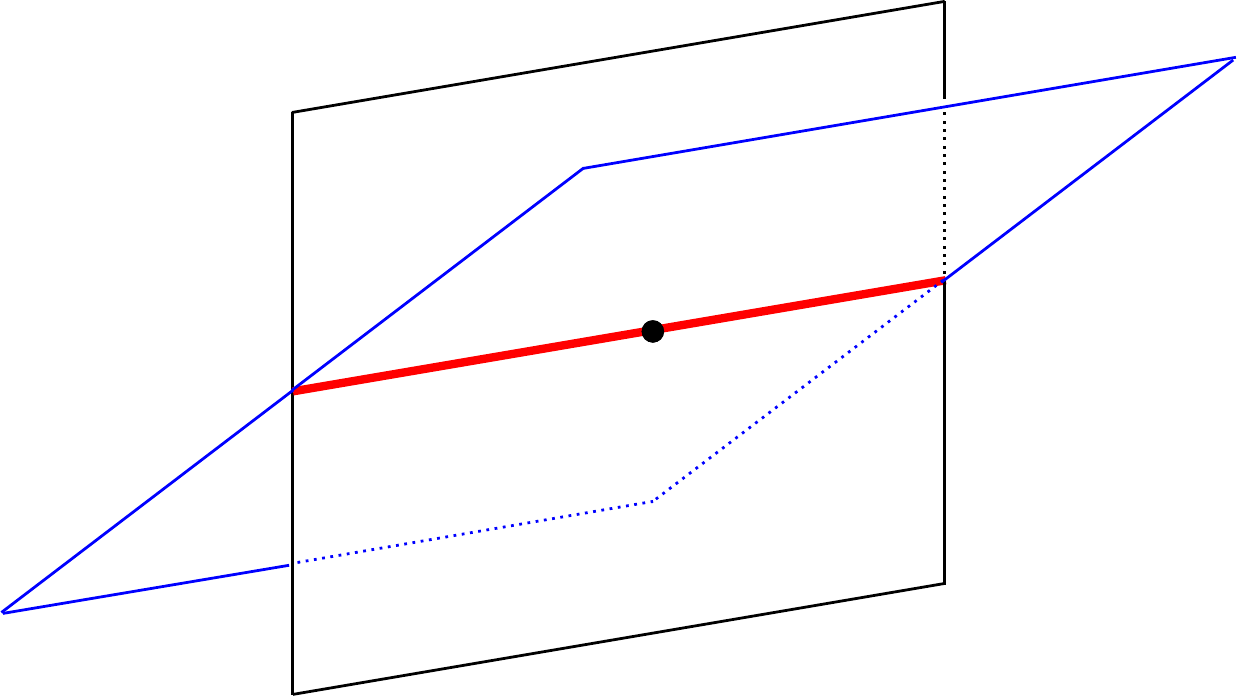}
\put(-150, 85){$\R Q_+$}
\put(3, 95){$\R Q_{\epsilon}$}
\\ a) & b) $\epsilon>0$ \\ \\
\includegraphics[height=4cm, angle=0]{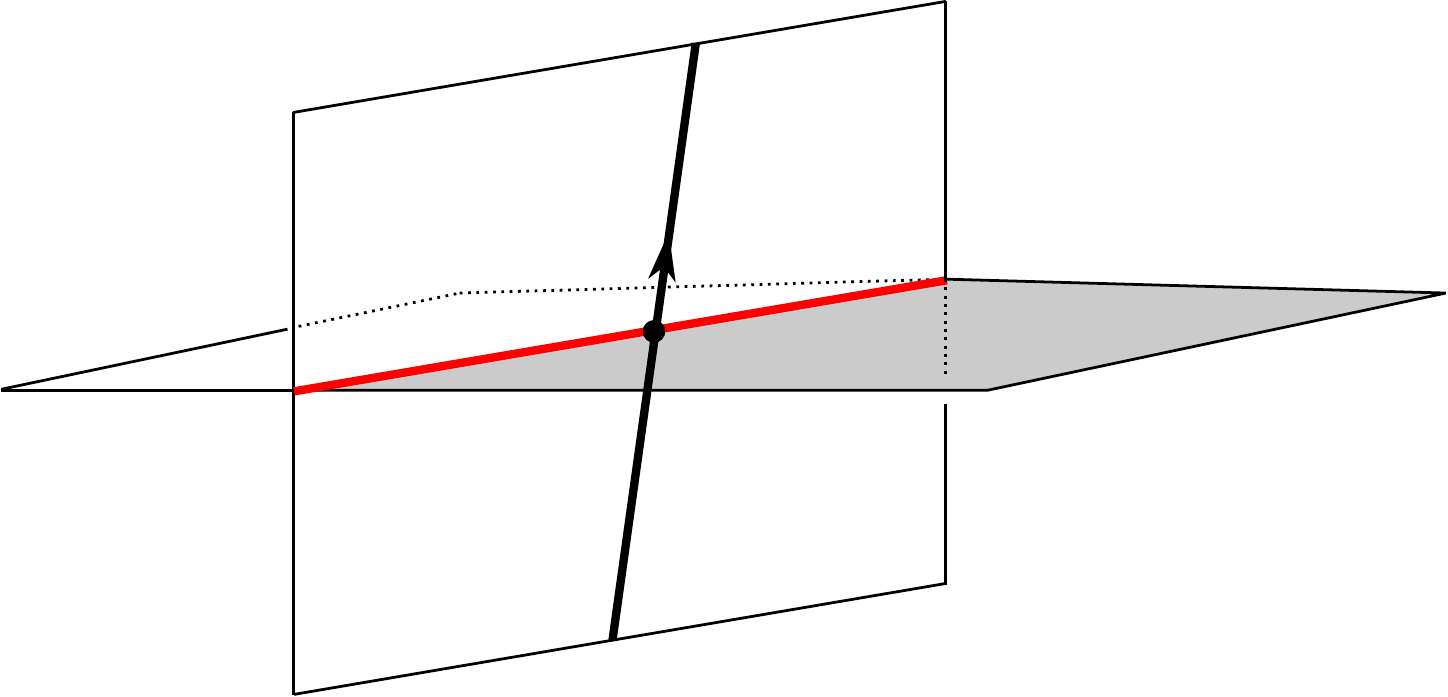}
\put(-120, 85){$f(\R P^1)$}
\put(-70, 55){$\Pi$}
\put(-270, 55){$\R\N_{\C| u_0}$}
& \includegraphics[height=4cm, angle=0]{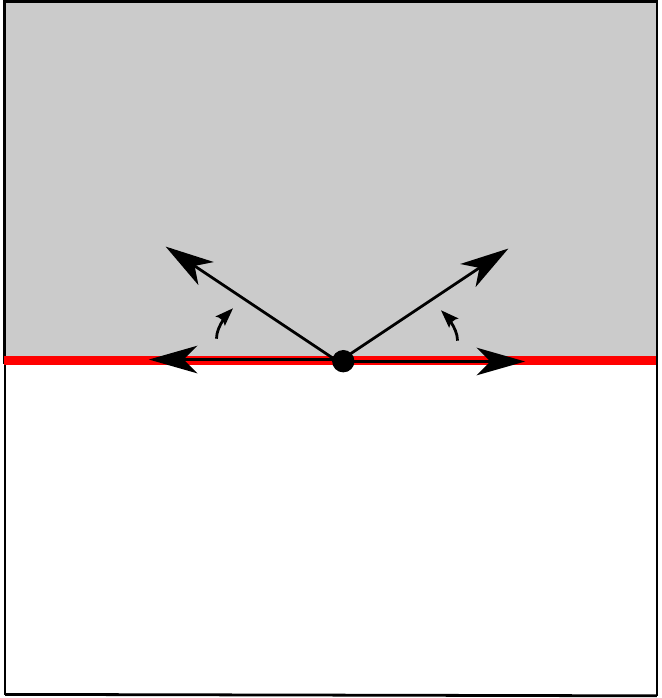}
\put(-90, 40){$\sigma(u_0)$}
\put(-90, 80){$\sigma(u)$}
\put(-40, 40){$\sigma(u_0)$}
\put(-40, 80){$\sigma(u)$}
\put(-100, 100){$\Pi$}
\put(5, 5){$\R\N_{\C| u_0}$}
\put(5, 55){$\R\N'_{\C| u_0}$}
\\
\\  c) & d)
\end{tabular}
\caption{Determining the  isotopy class of $\R\N'_\C$}
\label{fig:L or not}
\end{figure}

\medskip
As in the proof of Proposition \ref{prop:balanced2}, 
the deformation $p_\epsilon$ of $p_0$ as  an intersection point of 
$f_\epsilon(\CP^1)$ and $C_0$ is determined by the condition that
$f_{\epsilon}(\CP^1)\cap C_0$ has to realise the class 
$(a-b)E_{1,\epsilon} + bh$ in $\Pic_{2d}(C_0)$.
Since the class $h$ is constant, Inequality $(\ref{equ:diff})$
implies  that the direction of the vector
$$\frac{d p_{\epsilon}}{d \epsilon}|_{\epsilon=0}$$
is the same as for the line $D_1$ if $a>b$ and opposite if $a<b$.
Since the statement of the proposition holds for the classes $(1,0)$
and $(0,1)$, the proof is complete. 
\end{proof}

\begin{cor}\label{prop:W2-W3}
Let $f:\CP^1\to \CP^3$ be a real algebraic 
immersion such that $f(\CP^1)$ is contained in~$Q$, where it
has a bidegree
$(a,b)$ in the positive basis, with $a\ne b$. Then, one has
$$s_{\RP^3}(f(\CP^1))= \left\{\begin{array}{ll}
s_{\RP^1\times\RP^1}(f(\CP^1)) + b+1\,,& \mbox{if $a<b$,}
\\
\\  s_{\RP^1\times\RP^1}(f(\CP^1)) +b\,,& \mbox{if $a>b$.}

\end{array}\right.$$

\end{cor}

\section{Proofs of Theorems    \ref{thm:main} and \ref{thm:kollar}}\label{sec:proof}
In this section we recall in  Proposition \ref{prop:kollar}
the main statement for our purposes from \cite{Kol14} and
we  prove
  transversality  results needed to deduce Theorem
\ref{thm:kollar} from Proposition~\ref{prop:kollar}. We finish the section with the proof of Theorem \ref{thm:main}.\\

Let $C_0$ be a non-degenerate elliptic curve of degree 4 in $\CP^3$ and   
$\x\subset C_0$ be a configuration of $2d$ distinct points. We denote 
by $\Q$ the pencil of quadrics induced by $C_0$ and by $\overline \CC(\x)$ the 
set of connected algebraic curves of arithmetic genus 0 and 
degree $d$ in $\CP^3$ that contain $\x$.
Recall that $h\in\Pic_4(C_0)$ denotes the hyperplane section.

\begin{prop}[{\cite[Proposition 3]{Kol14}}] \label{prop:kollar}
Suppose that the points in $\x$ are in general position in $C_0$.
Then, every curve in $\overline \CC(\x)$ is irreducible and 
contained in a quadric of~$\Q$. Furthermore, 
the quadrics of $\Q$, that contain a curve in $\overline \CC(\x)$, are exactly 
the images
under $\pi_\Q$ of the solutions $E\in\Pic_2(C_0)$ of the equation
\begin{equation}\label{equ:pencil}
(d-2a)E=  (d-a)h -\x,
\end{equation}
with $0\le a<\frac{d}2$. \\

\noindent
If $Q$ is such a quadric and   $C$ is 
a curve in $\overline \CC(\x)$, then $C$ is linearly equivalent in $Q$
to $aD_1 + (d-a)D_2$, where 
  $D_1$ (resp. $D_2$) is a line in $Q$ whose intersection with $C_0$ is $E$
  (resp. $h-E$). Conversely, any irreducible rational curve $C$ in
$Q$, linearly equivalent to $aD_1+ (d-a)D_2$ and containing $2d-1$ of
the points in $\x$, is in $\overline \CC(\x)$ (i.e. contains $\x$).
\end{prop}

Any two solutions of Equation $(\ref{equ:pencil})$ differ by a torsion
point of order $d-2a$; in particular, Equation $(\ref{equ:pencil})$
has exactly $(d-2a)^2$ solutions in $C_0$.
 Let $\mathcal
M^*_{0,2d}(\CP^3,d)$ be the space of stable maps 
$f\!:\!(\CP^1,x_1,\ldots,x_{2d})\to\CP^3$
from $\CP^1$ with $2d$ marked points $x_1,\ldots,x_{2d}$, considered up to
reparametrisation, whose image has degree $d$. 
The evaluation map $ev$
is defined as 
$$\begin{array}{cccc}
ev:& \mathcal M^*_{0,2d}(\CP^3,d) &\longrightarrow & (\CP^3)^{2d}
\\  & f &\longmapsto & (f(x_1),\ldots,f(x_{2d})).
\end{array}$$

To complete the proof of Theorem \ref{thm:kollar}, it remains to 
prove that
one can choose $\x\subset C_0$ so that 
the map $ev$ is regular at every curve in $\overline
\CC(\x)$ and that each quadric of $\Q$, solution to Equation
$(\ref{equ:pencil})$, contains exactly 
$GW_{\C P^1\times\C  P^1}\left(a,d-a\right)$  elements of $\overline
\CC(\x)$.
This is done in the next two propositions. 
Denote by
$$
V_{n}\subset C_0^{n}\subset (\CP^3)^{n}
$$
 the set of
configurations of $n$ distinct points on $C_0$.

\begin{prop}\label{prop:generic quadric}
Let $Q$  be a quadric in $\Q$.
Given   an integer $a\in\{0,\ldots, d\}$ and
 $\y$ in $V_{2d-1}$, we denote by
$\overline \CC'_{Q,a}(\y)$ the set  of stable maps
 $f:(C,x_1,\ldots,x_{2d-1})\to Q$,
with $C$ a connected nodal curve  of arithmetic genus 0, such
that $f(C)$ has bidegree $(a,d-a)$ and
$f(\{x_1,\ldots,x_{2d-1}\})=\y\in V_{2d-1}$.
Then, there exists a dense open subset $U_{2d-1}\subset V_{2d-1}$ such that 
for every  choice of  $a$ and $\y\in U_{2d-1}$,
and for every stable map $f:(C,x_1,\ldots,x_{2d-1})\to Q$ in 
$\overline \CC'_{Q,a}(\y)$, one has
\begin{itemize}
\item $C$ is non-singular,
\item $f$ is an immersion.
\end{itemize}

\noindent
In particular, for every configuration $\y\in U_{2d-1}$, 
the quadric $Q$
contains exactly
$ GW_{\C P^1\times\C  P^1}(a,d-a)$  rational curves of bidegree $(a,d-a)$ and passing through $\y$.
\end{prop}
\begin{proof}
Let  $\mathcal V_a$ be the set of irreducible nodal rational curves 
in $Q$ of bidegree $(a,d-a)$. 
This is a quasiprojective subvariety of
dimension $2d-1$
of
the linear system $|aD_1 + (d-a)D_2|$ and we denote by
$\overline{\mathcal V}_a$
 its Zariski closure. We also denote by 
$\mathcal U_a$ the linear system on $C_0$ defined by the restriction  to
 $C_0$ of
 the divisor $aD_1 + (d-a)D_2$. 
Since $C_0$ has bidegree $(2,2)$ in $Q$, the linear system $\mathcal
U_a$ has degree $2d$.
By the Riemann-Roch theorem, every element of
$\mathcal U_a$
 is
determined by $2d-1$ of its points in $C_0$. 
In particular, every element $\y$ of $V_{2d-1}$ induces an element
$[\y]$ of $ \mathcal U_a $.
Since $C_0$ cannot be a component of a curve in 
$\overline{\mathcal V}_a$, the map
$$\begin{array}{cccc}
\phi: & \overline{\mathcal V}_a&\longrightarrow & \mathcal U_a 
\\ & C_1 &\longmapsto & C_1\cap C_0
\end{array}$$
 is well defined, generically finite, and 
 $\dim \phi(\overline{\mathcal
  V}_a\setminus\mathcal V_a)\le 2d-2$. 
Hence if $\y\in V_{2d-1}$ is
 so that 
$[\y]\notin\phi(\overline{\mathcal
  V}_a\setminus\mathcal V_a)$, then for every 
element $f:(C,x_1,\ldots,x_{2d-1})\to Q$ of
$\overline \CC'_{Q,a}(\y)$, the curve $f(C)$ must be a nodal irreducible
rational curve. In other words, we have $C=\CP^1$ and $f$ is an immersion.

Let $\mathcal
M^*_{0,2d-1}(Q,(a,d-a))$ be the space of stable maps 
$f:(\CP^1,x_1,\ldots,x_{2d-1})\to Q$
from $\CP^1$ with $2d-1$ marked points, considered up to
reparametrisation, whose image  
has bidegree $(a,d-a)$.
Since $Q\simeq \CP^1\times\CP^1$ is convex, the   proof of {\cite[Lemma
    1.3]{Wel2}} implies that  a   point $f$ in 
$\mathcal M^*_{0,2d-1}(Q,(a,d-a))$ is regular
 for the
corresponding evaluation map if and only if it is  an immersion. This completes the proof.
\end{proof}

\begin{prop}\label{prop:balanced}
Regular values of $ev$ contained in $V_{2d}$ form
 a dense open subset $U_{2d}\subset V_{2d}$.
In particular, if 
$\x\in U_{2d}$, 
then the set $\overline\CC(\x)$
contains exactly $GW_{\CP^3}(d)$ elements.
\end{prop}
\begin{proof}
By {\cite[Lemma 1.3]{Wel2}}, an element of $\mathcal
M^*_{0,2d}(\CP^3,d)$ is a regular point of $ev$ if and only if it is
a balanced immersion.
Let $\x$ be a configuration of $2d$ points on $C_0$ for which the
conclusions of Proposition \ref{prop:kollar} hold. 
Note that the quadrics,  solutions to Equation    $(\ref{equ:pencil})$,
do not change if~$\x$ is replaced by a configuration of points
linearly equivalent to $\x$ in $C_0$.
By the Riemann-Roch theorem, every effective divisor linearly equivalent to $\x$ in
$C_0$ is
determined by $2d-1$ points in $C_0$. Hence, according to Proposition~\ref{prop:generic quadric}, if $\x$ is chosen generically among all
configurations of $2d$ points in $C_0$ linearly equivalent to a given generic
divisor of degree $2d$, all elements in $\overline \CC(\x)$
will be immersions. 
Thus, we are left to show that
every element
$f:(\CP^1,x_1,\ldots,x_{2d})\to\CP^3$ of  $\overline \CC(\x)$ is
balanced, which follows from Proposition~\ref{prop:balanced2}.
\end{proof}

\noindent
 {\bf \emph{Proof of Theorem \ref{thm:kollar}.}} 
Let $U_{2d}$ be as in Proposition \ref{prop:balanced}. Choose $\x\in U_{2d}$ and $\y\subset\x$ a set of $2d-1$ points
in $\x$.
We denote by $\mathcal W_a $ the set of quadrics in $\Q$ corresponding to a
solution of Equation~$(\ref{equ:pencil})$. As explained above, this set
contains $(d-2a)^2$ elements.
By Proposition \ref{prop:kollar}, we know that the equality 
$$|\overline \CC(\x)|=\sum_{0\leq a<b} \ \sum_{Q\in \mathcal W_a}|\overline
 \CC'_{Q,a}(\y)|$$ 
holds. By Proposition
\ref{prop:balanced}, we have  $|\overline \CC(\x)|=GW_{\CP^3}(d)$.
Since $\x$ is a regular value of the map $ev$, every map in $\overline \CC(\x)$
is an immersion $f:\CP^1\to \CP^3$. 
Hence by Proposition \ref{prop:generic quadric}, we have $|\overline
\CC'_{Q,a}(\y)|=GW_{\CP^1\times\CP^1}(a,b)$.~ $\hfill\qed$\\

\noindent
 {\bf \emph{Proof of Theorem \ref{thm:main}.}}
Suppose now that $C_0$ is real, with $\R C_0\ne 0$, and   
$\x\in U_{2d}$ is a
real configuration of $2d$ points, containing at
least one real point.

Let $d$ be an odd positive integer and   $Q$ be a quadric
containing a real map $f\!:\!\CP^1\to\CP^3\in \overline\CC(\x)$. 
According to Proposition
\ref{prop:kollar}, there exists an integer
$a\in\{0,\ldots,\frac{d-1}{2} \}$ such that
$f(\CP^1)$ 
has bidegree either $(a,d-a)$ or $(d-a,a)$ in the positive basis of
$Q$. From Corollary
\ref{prop:W2-W3}, we deduce that
$$s_{\RP^3}(f(\CP^1))= \left\{\begin{array}{ll}
s_{\RP^1\times\RP^1}(f(\CP^1))\,,& \mbox{if $a$ is even,}
\\
\\  s_{\RP^1\times\RP^1}(f(\CP^1)) +1\,,& \mbox{if $a$ is odd.}

\end{array}\right.$$

 Thus, the total contribution to $W_{\RP^3}(d,l)$ of elements 
 of $\overline \CC(\x)$, whose image is contained in $Q$, is
$W_{\RP^1\times\RP^1}((a,d-a),l)$ if $a$ is even and 
$-W_{\RP^1\times\RP^1}((a,d-a),l)$ if $a$ odd.
Equation~$(\ref{equ:pencil})$ always has a real solution and 
two real solutions differ by a real torsion element of order $d-2a$.
Hence, Equation~$(\ref{equ:pencil})$ has
exactly
 $d-2a$ real
solutions for every $a\in\{0,\ldots,\frac{d-1}{2}\}$. Now the end of the proof is similar to the proof of Theorem 
\ref{thm:kollar}. $\hfill\qed$

\begin{rem}\label{rem:vanishing}
One can also prove the vanishing of  $W_{\RP^3}(d,l)$ for $d$ even and $l\le
d-1$ in this way. 
This is also how J. Koll\'ar exhibited configurations $\x$ for which
$ \overline \CC(\x)$ contains no real curves: 
if $h$ is on the pointed component of $\R C_0$ and $\x$ is on
the non-pointed component, then  Equation $(\ref{equ:pencil})$ has no
real solution.
\end{rem}

\section{Computations and further comments}\label{sec:compute}

We provide in Tables \ref{tab1} and \ref{tab2} 
the  values of $W_{\RP^3}(d,l)$ for small values of
$d$. The values $W_{\RP^3}(d,d)$ are taken from \cite{GeorZin13}.
Our values
of $W_{\RP^3}(d,0)$ agree with the ones computed in \cite{Br7} and \cite{Br8}.

\begin{table}[h]
$$\begin{array}{|c|r|r|r|r|r|r|r|r|r|}
\hline \mbox{\backslashbox{l}{d}} & 1&3&5&7&9 &11& 13 
\\ \hline
0 & 1 &  -1 & 45& -14589& 17756793 &-58445425017 
&426876362998821
\\\hline 1 & 1 & -1 & 29 & -6957 & 6717465 & -18318948633&
114201657733941
\\\hline 2 & & -1 & 17 & -3093 &2407365  & -5495423913 &
29447853240537
\\\hline 3 & & -1 & 9 & -1269& 812157 & -1571343273 & 7298043143697
\\\hline 4 & & & 5 & -477 &  256065 & -426170217 & 1732456594269
\\\hline 5 &  &  & 5 & -173 &   75281 &  -109136649 & 392521356477
\\\hline 6 &  &  & & -85  & 21165 & -26389305 & 84651531633
\\\hline 7 &  & &  & -85  &6165 & -6109369 &  17390628729
\\\hline 8 &  & &  &  & 1993 & -1401241 & 3432362709
\\\hline 9 &  & &  &  & 1993 & -336441 &  663105669 
\\\hline 10 & & &  &  & & -136457 & 129344841
\\\hline 11 & & &  &  & & -136457 &  27607073
\\\hline 12 & & &  &  & & & 3991693
\\\hline 13 & & &  &  & & & 3991693
\\\hline 
\end{array}$$
\caption{}
\label{tab1}
\end{table}

\begin{table}[h]
$$\begin{array}{|c|r|r|}
\hline \mbox{\backslashbox{l}{d}} &  15 &17 
\\ \hline
0 & -6061743911446054965 & 152244625648721441783409 
\\\hline 1 & -1414422922125979269 & 31497207519483035166897
\\\hline 2 & -319737783634469757 & 6337510847893018140813
\\\hline 3 & -69876860779936989 & 1238195460245786397189
\\\hline 4 & -14727767907263157 & 234469282186353521817
\\\hline 5 &  -2985647746084965 & 42946188374781866313
\\\hline 6 &  -580664589588189 & 7592707791183642453
\\\hline 7 &  -108170761670685 & 1293343577697132477
\\\hline 8 & -19320554509557 &  212071309052944257
\\\hline 9 &  -3327374698245 & 33506171960522913
\\\hline 10 & -558961586685 & 5121214631258589
\\\hline 11 & -93320976413 & 763120829396277
\\\hline 12 &  -16000904949 & 112222758491433
\\\hline 13 &  -2937725541 & 16596074817721
\\\hline 14 & -1580831965 & 2542297019941
\\\hline 15 &  -1580831965&  447392666733
\\\hline 16 &  & -129358296175
\\\hline 17 &  & -129358296175
\\\hline 
\end{array}$$
\caption{}
\label{tab2}
\end{table}

The tables show that the values of  the Welschinger invariants with 2 or 0 real points
 are the same in the computed degrees. This lead us to make the following conjecture.
\begin{conj}\label{prop:2-no real}
For every positive integer $d$, one has
$$W_{\R  P^3}(d,d-1)=W_{\R  P^3}(d,d).$$
\end{conj}

\noindent
From the first values of $W_{\RP^3}(d,l)$, it might be tempting to
conjecture
that 
$$(-1)^{k} W_{\R  P^3}(2k+1,l)\ge
(-1)^{k} W_{\R  P^3}(2k+1,l+1)\ge 0.$$
However, both inequalities turn out not to hold in general, starting in degrees 17
and 19:
$$W_{\R  P^3}(17,16)<0\quad\mbox{and}\quad  
-W_{\R  P^3}(19,17)=74131154312945<106335656443537 = -W_{\R  P^3}(19,18).$$
In fact, the sign of $W_{\R  P^3}$ seems to obey  an
analogous  rule as the one observed in \cite{Bru14} for
the Welschinger invariants of $\RP^2$:
as $l$ goes from $0$ to $d-1$, the numbers $(-1)^kW_{\RP^3}(2k+1,l)$ are
first positive and then, starting from some mysterious threshold,
have an alternating sign. It would be interesting to investigate what
governs the sign of the Welschinger invariants.

\medskip
We end this paper by establishing a congruence modulo 4 between
the Welschinger and the Gromov-Witten invariants, generalising results from  
\cite{Br7} and \cite{GeorZin13}.
\begin{prop}\label{prop:mod W2}
For every positive integers $a$, $b$, and $l\in\{0,\ldots, a+b-1\}$, one has
$$GW_{\C P^1\times\CP^1}(a,b)= W_{\R P^1\times\RP^1}((a,b),l) \mod 4. $$
\end{prop}
\begin{proof}
 Both invariants can be computed via the
enumeration of marked floor diagrams (see \cite{Br6b})
with  Newton polygon the rectangle with
vertices
$$(0,0), \ (a,0), \ (0,b), \ \mbox{and}\ (a,b). $$
Any floor diagram with this Newton polygon has only
floors with divergence 0. Hence, every marked floor diagram
contributing to 
$GW_{\C P^1\times\CP^1}(a,b)$    has only  non-negative real multiplicities. 
Now the result follows from the fact that the absolute value of a real
multiplicity of a marked floor diagram is always equal modulo 4 to its
complex multiplicity.
\end{proof}

\begin{cor}
For every  positive integers $d$ and $l\in\{0,\ldots,d\}$, one has
$$GW_{\C P^3}(d)= (-1)^{\frac{(d-1)(d-2)}2} W_{\R P^3}(d,l) \mod 4.$$
\end{cor}
\begin{proof}
When $d$ is even, this follows from Theorem \ref{thm:kollar} and the
vanishing of $W_{\R P^3}(d,l) $. 
When $d$ is odd, this follows from Theorem \ref{thm:main}, Proposition \ref{prop:mod W2}, 
and the congruence $(-1)^{u}=2u+1 \mod 4 $.
\end{proof}
In \cite{Br7} it was  also shown that the inequality $(-1)^{\frac{(d-1)(d-2)}2} W_{\R
  P^3}(d,0)\ge 0$   holds for all  $d$. It is not clear how to deduce this   inequality
from Theorem \ref{thm:main}.

\bibliographystyle {alpha}
\bibliography {Biblio.bib}

\end{document}